\documentclass[reqno]{amsart}

\usepackage{amsmath,amssymb,amscd, accents} \usepackage{graphicx}
\DeclareGraphicsExtensions{.eps} \usepackage{mathrsfs}
\usepackage[mathcal]{eucal} 
\usepackage{esint}
\usepackage[all]{xy}

\newtheorem{Thm}{Theorem}{\bfseries}{\itshape}
\newtheorem*{Thm*}{Theorem}{\bfseries}{\itshape}
\newtheorem{Cor}{Corollary}{\bfseries}{\itshape}
\newtheorem{Prop}[Cor]{Proposition}{\bfseries}{\itshape}
\newtheorem*{Prop*}{Proposition}{\bfseries}{\itshape}
\newtheorem{Lem}[Cor]{Lemma}{\bfseries}{\itshape}
\newtheorem*{Lem*}{Lemma}{\bfseries}{\itshape}
{\bfseries}{\itshape}
{\bfseries}{\itshape}
\newtheorem{Def}[Cor]{Definition}{\bfseries}{\rmfamily}
{\scshape}{\rmfamily}
\newtheorem{Rem}[Cor]{Remark}{\scshape}{\rmfamily}
{\bfseries}{\itshape}

\renewcommand\ge{\geqslant} \renewcommand\le{\leqslant}
\let\tildeaccent=\~ \let\hataccent=\^
\renewcommand\~[1]{\widetilde{#1}}

\def\<{\left<} \def\>{\right>} \def\({\left(} \def\){\right)}

\let\parasymbol=\S \def\secref#1{\parasymbol\ref{#1}}
 \def\pd#1#2{\tfrac{\partial#1}{\partial#2}}

\let\polishL=l \def\Zoladek.{\.Zol\c adek}

 \def\const{\operatorname{const}}

\def\Re{\operatorname{Re}} \def\Im{\operatorname{Im}}

\def\GL{\operatorname{GL}} \def\SL{\operatorname{SL}}

\def\PGL{\operatorname{PGL}}
\def\etc.{\emph{etc}.}

\def\:{\colon} \def\R{{\mathbb R}} \def\C{{\mathbb C}} \def\Z{{\mathbb
    Z}} \def\N{{\mathbb N}} \def\Q{{\mathbb Q}} \def\P{{\mathbb P}}
\def\H{{\mathbb H}}
\def\F{{\mathbb F}}

\def\A{{\mathbb A}}

\let\PolishL=\L 
\def\L{{\mathbb L}}

 \def\e{\varepsilon} \def\S{\varSigma}

\def\poly{{\operatorname{poly}}}

 \def\d{\,\mathrm d}

 \def\Lojas.{\PolishL ojasiewicz}
 
 \def\cR{{\mathcal R}}

\def\cF{{\mathcal F}}  \def\cR{{\mathcal R}}

\def\cO{{\mathcal O}}

\def\rest#1{{\vert_{#1}}}

\def\Gal{\operatorname{Gal}}

\def\id{\operatorname{id}}

\def\O{\mathcal{O}}

\def\alg{\mathrm{alg}}
\def\trans{\mathrm{trans}}

\def\disc{\operatorname{disc}}
\def\Pic{\operatorname{Pic}}

\def\^#1{^{(#1)}{}}

\def\tf{{\mathrm{tf}}}
\def\sp{{\mathrm{sp}}}
\def\spp{{\mathrm{spp}}}
\def\wsp{{\mathrm{wsp}}}

\def\^#1{^{(#1)}{}}

\begin{document}

\makeatletter
\let\@wraptoccontribs\wraptoccontribs
\makeatother

\title{Some effective estimates for Andr\'e-Oort in $Y(1)^n$}

\author{Gal Binyamini}
\thanks{G.B. is the incumbent of the Dr. A. Edward Friedmann career
  development chair in mathematics. This project has received funding
  from the European Research Council (ERC) under the European Union's
  Horizon 2020 research and innovation programme (grant agreement No
  802107). This research was supported by the ISRAEL SCIENCE
  FOUNDATION (grant No. 1167/17) and by funding received from the
  MINERVA Stiftung with the funds from the BMBF of the Federal
  Republic of Germany.}
      
\contrib[with an appendix by]{Emmanuel Kowalski}

\address{Weizmann Institute of Science, Rehovot, Israel}
\email{gal.binyamini@weizmann.ac.il}

\begin{abstract}
  Let $X\subset Y(1)^n$ be a subvariety defined over a number field
  $\F$ and let $(P_1,\ldots,P_n)\in X$ be a special point not
  contained in a positive-dimensional special subvariety of $X$. We
  show that the if a coordinate $P_i$ corresponds to an order not
  contained in a single exceptional Siegel-Tatuzawa imaginary
  quadratic field $K_*$ then the associated discriminant
  $|\Delta(P_i)|$ is bounded by an effective constant depending only
  on $\deg X$ and $[\F:\Q]$. We derive analogous effective results for
  the positive-dimensional maximal special subvarieties.

  From the main theorem we deduce various effective results of
  Andr\'e-Oort type. In particular we define a genericity condition on
  the leading homogeneous part of a polynomial, and give a fully
  effective Andr\'e-Oort statement for hypersurfaces defined by
  polynomials satisfying this condition.
\end{abstract}


\date{September 14, 2018}
\maketitle

\section{Introduction}

\subsection{Notations}

We identify $Y(1)\simeq \A_\Q^1$ be means of the $j$-invariant. A
point in $p\in Y(1)$ is said to be \emph{special} if the corresponding
elliptic curve $E_p$ admits complex multiplication. In this case we
denote by $K(p)$ the quadratic field generated by the periods $E_p$,
and by $\Delta(p)$ the discriminant of the endomorphism ring of $E_p$
in the ring of integers of $K$ (see~\secref{sec:special-points}).  A
point $P=(P_1,\ldots,P_n)\in Y(1)^n$ is called special if each
coordinate $P_i$ is special. We denote $D(P):=\max_i|\Delta(P_i)|$.

For every $N>1$ there is a \emph{modular polynomial}
$\Phi_N(x,y)\in\Z[x,y]$ whose zero locus in $Y(1)^2$ is the set of
pairs $(x,y)$ corresponding to $N$-isogenous elliptic curves. Let
$S_0\cup\cdots\cup S_w$ be a partition of $\{1,\ldots,n\}$ with $S_0$
only permitted to be empty. Let $p_i\in Y(1)$ be special for
$i\in S_0$. Let $s_i\in S_i$ be minimal for $i>0$ and for each
$s_i\neq j\in S_i$ choose a positive integer $N_{ij}$. Then a
$w$-dimensional \emph{special subvariety} $V$ of $Y(1)^n$ is an
irreducible component of the variety
\begin{multline}\label{eq:special-variety}
  \{ (y_1,\ldots,y_n)\in X : y_i=p_i, i\in S_0,\\
  \Phi_{N_{ij}}(y_{s_i},y_j)=0, s_i\neq j\in S_i, i=1,\ldots, w\}.
\end{multline}
Following \cite{pila:andre-oort}, we will call a special variety
\emph{strongly special} if $S_0=\emptyset$. We denote $P(V):=(p_i)_{i\in S_0}$
and set $D(V)=D(P(V))$. If $V$ is defined as above but without the
condition that $p_i$ is special we say that $V$ is
\emph{weakly-special}.

\subsection{Results for general varieties}

To state our main result, we let $K_*$ denote a universally fixed
quadratic field and $\Delta_*$ the discriminant of its ring of
integers. The field $K_*$ arises from the theorem of Siegel-Tatuzawa
and its definition is given in~\secref{sec:siegel-tatuzawa}. We note
that it is possible that such an exceptional field does not exist (for
instance, this is the case if one assumes the generalized Riemann
hypethesis for imaginary quadratic fields). In this case we may
formally set $K_*=\emptyset$ below, and our results then become a
fully effective solution of the Andr\'e-Oort conjecture for
$Y(1)^n$. We define $D_*(P)$ and by analogy with $D(P)$, but with the
maximum taken only over those $P_i$ where $K(P_i)\neq K_*$ (and
similarly for $D_*(V)$).

For a variety $X\subset Y(1)^n$ we denote by $\F_X$ its field of
definition. We denote by $X^\sp$ the set of all maximal special
subvarieties of $X$. We also denote by $X^\spp$ the set of all special
points in $X$ that are not contained in a positive-dimensional special
subvariety of $X$. We denote by $\deg X$ the degree of $X$ with
respect to the projective embedding $Y(1)^n\simeq\A^n\subset\P^n$
(more precisely we take the sum of the degrees of all irreducible
components). If $X$ is a hypersurface given as the zero locus of a
polynomial $F$ over a number field then we denote by $H(F)=H(X)$ the
maximal Weil height of any of its coefficients.

As a general convention throughout the paper we denote by $c(\cdots)$
some effectively computable constant depending on a set of
parameters. If the constant does not depend on any parameters we use
``$\const$'' to avoid confusion. Our main result is as follows.

\begin{Thm}\label{thm:main}
  Let $X\subset Y(1)^n$ be defined over a number field and let
  $V\in X^\sp$. Then
  \begin{align}
    \deg V &\le c(\deg X,[\F_X:\Q]), & D_*(V) &\le c(\deg X,[\F_X:\Q]).
  \end{align}
  For $D(V)$ we have the weaker estimate
  \begin{equation}
    D(V) \le c(\Delta_*,\deg X,[\F_X:\Q]).
  \end{equation}
\end{Thm}

To state some consequences of Theorem~\ref{thm:main} we introduce the
following terminology.  A special variety $V\subset Y(1)^n$ is called
a \emph{$*$-variety} if all associated quadratic fields of $P(V)$
equal $K_*$. A variety $X'$ is called a \emph{special section} of a
variety $X$ if
\begin{equation}
  X' = \pi(X\cap\{x_i=P_i\text{ for all } i\in\Sigma\}), \qquad \pi(x_1,\ldots,x_n)=(x_i)_{i\not\in\Sigma}
\end{equation}
for some $\Sigma\subset\{1,\ldots,n\}$ and some special points
$P_i\in Y(1)$ for $i\in\Sigma$.

Theorem~\ref{thm:main} reduces the problem of effectively computing
$X^\sp$ to the problem of computing the $*$-varieties in
$X_\alpha^\sp$ where $X_\alpha$ ranges over an effectively constructed
set of special sections of $X$. Indeed, for each set
$\Sigma\subset\{1,\ldots,n\}$ and each choice of $P_i$ with
discriminant at most\footnote{We remark that there exist only a finite
  number of special points in $Y(1)$ with a given discriminant, and it
  is straightforward to effectively enumerate them.}
$c(\deg X,[\F_X:\Q])$ define $X_\alpha$ to be the corresponding
special section. Then every $V\in X^\sp$ corresponds to a $*$-variety
in some $X_\alpha^\sp$.

The following corollary is a generalization of
\cite[Theorem~3,~Corollary~2]{kuhne:ao2} from $n=2$ to the general
case. It provides a uniform bound on the discriminants of special
points where at most one coordinate corresponds to $K_*$.

\begin{Cor}\label{cor:no-modular}
  Let $X\subset Y(1)^n$ be defined over a number field and let
  $V\in X^\sp$. Suppose that at most one of the associated quadratic
  fields of $P(V)$ equal $K_*$. Then
  \begin{equation}
    D(V) \le c(\deg X,[\F_X:\Q]).
  \end{equation}  
\end{Cor}

\subsection{Results for degree non-degenerate varieties}

In some cases Theorem~\ref{thm:main} can be combined with additional
arguments to produce a fully effective Andr\'e-Oort type statement. 
To state such a result we introduce the following terminology.

\begin{Def}[Degree non-degenerate polynomials]
  A polynomial $F\in\C[x_1,\ldots,x_n]$ is called \emph{degree
    non-degenerate} (or dnd) if for every $i=1,\ldots,n$ we have
  either $\deg_{x_i} F=\deg F$ or $\deg_{x_i}F=0$. The zero locus of
  such $F$ is called a dnd hypersurface.

  We say that $F$ is \emph{hereditarily dnd} (or hdnd) if it is dnd,
  and if for each $i\neq j$ the restriction $F\rest{x_i=x_j}$, viewed
  as a polynomial in $n-1$ variables, is hdnd. The zero locus of such
  $F$ is called an hdnd hypersurface.
\end{Def}

Our key technical result concerning dnd hypersurfaces is as follows.

\begin{Cor}\label{cor:dnd}
  Let $X\subset Y(1)^n$ be a dnd hypersurface defined over a number
  field and let $P\in X^\spp$. Then either $P_i=P_j$ for some
  $i\neq j$, or
  \begin{equation}\label{eq:dnd}
    D(P) \le c(\deg X,H(X),[\F_X:\Q]).
  \end{equation}
\end{Cor}

In the case that $X$ is hdnd we derive from Corollary~\ref{cor:dnd}
the following effective Andr\'e-Oort statement. We call a special
variety \emph{linear} if the modular relations $\Phi_N(x_i,x_j)$
involved in its definition are all of the form $x_i\equiv x_j$.

\begin{Cor}\label{cor:hdnd}
  Let $X\subset Y(1)^n$ be an hdnd hypersurface defined over a number field
  and let $V\in X^\sp$. Then $V$ is linear and
  \begin{equation}
    D(V) \le c(\deg X,H(X),[\F_X:\Q]).
  \end{equation}
\end{Cor}

Corollary~\ref{cor:hdnd} provides a wealth of examples in arbitrary
dimension and degree where $X^\sp$ can be effectively computed ---
indeed an open dense set of examples in every dimension and degree.
In particular Corollary~\ref{cor:hdnd} applies when $X$ is a linear
hypersurface. More generally Corollary~\ref{cor:dnd} gives another
proof of the key technical ingredient \cite[Lemma~3]{bk:linear} in the
recent effective solution of Andr\'e-Oort for arbitrary linear
subvarieties due to Bilu-K\"uhne. Our proof of Corollary~\ref{cor:dnd}
is based on an idea inspired by \cite{bk:linear} involving asymptotics
around the cusp, but modulo Theorem~\ref{thm:main} the argument
becomes quite simple.

\begin{Rem}
  It is intructive to consider an example where the conclusion of
  Corollary~\ref{cor:hdnd} fails. Perhaps the simplest such example is
  the modular polynomial $\Phi_2$ given by
  \begin{multline}
    \Phi_2(x,y):=x^3+y^3-x^2y^2+1488xy(x+y)-162\cdot10^3(x^2+y^2)+40773375xy\\+
    8748\cdot10^6(x+y)-157464\cdot10^9.
  \end{multline}
  As guaranteed by Corollary~\ref{cor:hdnd}, the 2-modular curve is not
  hdnd since we have the inequality $\deg_x\Phi_2=3<4=\deg\Phi$.
\end{Rem}

\subsection{Overview of the proof}

We start by reviewing the proof of Theorem~\ref{thm:main}. The proof
follows the general approach of Pila \cite{pila:andre-oort}. We assume
for the purposes of this overview that the reader is familiar with
this approach, as there are already several good surveys available
(e.g. \cite{scanlon:survey}) in addition to the original paper. There
are three main sources of ineffectivity in the proof of
\cite{pila:andre-oort}, as follows:
\begin{enumerate}
\item The Siegel bound that is used to produce lower bounds
  $h(d)\gg |d|^{1/2-\e}$ for the class number $h(d)$ is ineffective.
\item The Pila-Wilkie bound for sets definable in
  $\R_{\mathrm{an},\exp}$, which is used to provide a competing upper
  bound for Galois orbits, is ineffective.
\item The process by which one reduces the problem of controlling the
  maximal special varieties $X^\sp$ to the problem of controlling the
  maximal special points $X^\spp$ employs o-minimal finiteness
  properties and is ineffective.
\end{enumerate}
In~\secref{sec:proof-Xspp} we review the effectiviation of the upper
bound for special points in $X^\spp$, corresponding roughly to items
1--2 above. In~\secref{sec:proof-Xsp} we review the main idea used to
effectively reduce the computation of $X^\sp$ to that of
$X^\spp$. Finally in~\secref{sec:proof-hdnd} we review the proof of
the fully effective result for hdnd hypersurfaces.

\subsubsection{Effectivizing the upper bound for $P\in X^\spp$}
\label{sec:proof-Xspp}

To deal with the ineffectivity of the lower bound we appeal to a
result of Tatuzawa \cite{tatuzawa} stating that the constant in
Siegel's bound can be made effective for all discriminants except
those corresponding to orders in a single imaginary quadratic field
$K_*$. The possibility of using this result was already mentioned in
\cite[Section~13.3]{pila:andre-oort}, where it was noted that if one
could effectivize the requisite Pila-Wilkie statement this would lead
to a bound for the number of points in $X^\spp$ whose coordinates all
correspond to fields other than $K_*$. Note however that we produce
bounds for the discriminants of non-exceptional coordinates even when
some coordinates do correspond to $K_*$. This extra generality does
not follow straightforwardly since the height of such points is
determined by the coordinates of largest discriminant, which may be
associated to $K_*$.

The main result of \cite{me:noetherian-pw} can be used to effectivize
the Pila-Wilkie bounds needed in Pila's proof with a significant
caveat: the results apply to any compact subdomain of the fundamental
domain, but cannot be used to obtain effective bounds uniformly over
the entire (non-compact) fundamental domain. To overcome this
difficulty we appeal to Duke's equidistribution theorem
\cite[Theorem~1.i]{duke:equidist}. This implies that the Galois orbit
of each coordinate $P_i$ is equidistributed in $Y(1)$, and in
particular that a large portion of it is contained in a fixed compact
subset. Duke's result is itself ineffective because of the same use of
Siegel's ineffective bound. It can however be made effective using
Tatuzawa's theorem, with the same exception for discriminants of
$K_*$. A sketch of the proof of this variant of Duke's Theorem is
provided for completeness in Appendix~\ref{appendix} by E. Kowalski.

Let $K_*^\tf$ denote the field generated by all ring class fields
associated to $K_*$. Our idea for dealing with the case that some
coordinates of $P$ correspond to $K_*$ is to repeat Pila's argument
\emph{relativized} over $K_*^\tf$. More specifically, we use a result
of Cohn \cite[Proposition~8.3.12]{cohn:class-fields} showing that
$K_*^\tf$ is almost disjoint from the ring class fields of non-$K_*$
discriminants. We show that this implies that the number of conjugates
of $P$ under $\Gal(\bar\Q/K_*^\tf)$ remains large, and that a large
part of this orbit remains in a compact subset as predicted by the
equidistribution theorem. One can therefore \emph{fix} the values of
all $K_*$-coordinates and apply Pila's strategy to the resulting
section, obtaining an upper bound for all discriminants of non-$K_*$
coordinates. A crucial point here is that the effective Pila-Wilkie
theorem of \cite{me:noetherian-pw} gives constants that are
independent of the chosen section.

\subsubsection{Effectivizing the upper bound for $X^\sp$}
\label{sec:proof-Xsp}

As in Pila's appraoch, our idea is to inductively reduce the study of
$X^\sp$ to the study of $X^\spp$ by analyzing the possible types of
the positive-dimensional maximal special subvarieties of $X$. The new
ingredient in our approach is a transformation of this question, via a
differential algebraic construction, to a completely algebraic
question. The reduction is based on the following observation: while
the special subvarieties are defined by algebraic conditions
$\Phi_N(x_i,x_j)=0$ of (a-priori) unbounded degrees, their preimages
under the universal covering map $\pi:\H^n\to Y(1)^n$ satisfy
essentially linear relations $\tau_i=g\cdot\tau_j$ where
$g\in\GL_2^+(\Q)$. After describing the graph of $\pi$ using a
rational vector field encoding the differential equation satisfied by
the $j$-function, the problem reduces to describing the points where
the trajectory of a vector field satisfies such a linear condition
identically.

The problem above is almost amenable to methods of differential
equations. In particular it may be studied using \emph{multiplicity
  estimates} for the maximal order of vanishing of a polynomial on the
trajectory of a vector field. More precisely, the problem can be
reduced to a purely algebraic problem if we relax the condition
$g\in\GL_2^+(\Q)$ to $g\in\GL_2(\C)$. Luckily the functional
transcendence results of \cite{pila:andre-oort} can be used to show
that the existence of any such algebraic dependence actually implies
the existence of a dependence with rational coefficients.

We remark that the methods applied here yield explicit and fairly
sharp bounds, and can be used to treat similar problems in far greater
generality, including similar questions in the context of abelian
varieties (treated by other methods in \cite[Theorem~1.a]{bz:heights})
and Shimura varieties more general than $Y(1)$. In a joint work in
progress with Christopher Daw we apply these ideas to the study of
optimal varieties in the sense of the Zilber-Pink conjecture for these
various contexts.

\subsubsection{The effective result for hdnd hypersurfaces}
\label{sec:proof-hdnd}

In light of Theorem~\ref{thm:main}, one can essentially reduce to the
case of finding an upper bound for the $*$-points on an hdnd
hypersurface. In other words we may assume that all coordinates are
associated to the fundamental discriminant $\Delta_*$, say
$\Delta_i=f_i^2\Delta_*$. Assume without loss of generality that
$|\Delta_1|$ is maximal among these.

We use a simple idea borrowed from \cite{kuhne:ao2,bk:linear}:
conjugating $P_1$ to the point closest to the cusp within its orbit,
we see from the asymptotic expansion of the $j$-function around the
cusp that $P_1\simeq e^{\pi f_1\sqrt{|\Delta_*|}}$. Any point $P_i$
with $f_i<f_1$ is similarly majorated by
$e^{\pi f_i\sqrt{|\Delta_*|}}$, while any point $P_i\neq P_1$ with the
same discriminant ($f_i=f_1$) turns out to be majorated by
$e^{(1/2)\pi f_1\sqrt{|\Delta_*|}}$. Thus, if no point $P_i$ equals
$P_1$ we see that $P_1^n$ becomes asymptotically dominant over all
other terms in our hdnd equation assuming that $\Delta_*$ is
sufficiently large. If we assume on the contrary that $\Delta_*$ is
bounded by some (effective) constant then Theorem 1 already becomes
fully effective. This finishes the proof of
Corollary~\ref{cor:dnd}. Corollary~\ref{cor:hdnd} is then proved using
an induction over the dimension by intersecting with all possible
diagonals $x_i=x_j$.

\subsection{Acknowledgements}

I would like to express my gratitude to Jonathan Pila for discussions
regarding effectivity issues surrounding Andr\'e-Oort, and in
particular for pointing me in the direction of Tatuzawa's result; to
Elon Lindenstrauss for suggesting Duke's equidistribution result as a
potential remedy for the compactness condition in my paper
\cite{me:noetherian-pw}; and to Gabriel Dill and the anonymous referee
for some corrections and suggestions on the initial version of the
manuscript. I also thank the Tokyo Institute of Technology for their
hospitality during a visit in which some of this work was carried out.

\section{Auxiliary results}

\subsection{Special points on $Y(1)$ and their discriminants}
\label{sec:special-points}

Recall that we identify $Y(1)\simeq \A_\Q^1$ be means of the
$j$-invariant. A point $p\in Y(1)$ is special if and only if
$p=j(\tau)$ for some quadratic number $\tau\in\H$. In this case we
denote by $K(p):=\Q(\tau)$ the quadratic field generated by $\tau$, by
$\cO_{\Q(\tau)}$ the maximal order of $\Q(\tau)$, and by $\cO_\tau$
the order of $\Q(\tau)$ given by the endomorphism ring of
$\Z[\tau]$. The \emph{discriminant} of $\tau$ is defined to be
\begin{equation}
  \Delta(\tau) := \disc \cO_\tau = f_\tau^2 \disc \cO_{\Q(\tau)}
\end{equation}
where $f_\tau$ is the conductor of $\cO_\tau$ with respect to
$\cO_{\Q(\tau)}$. We have $\cO_\tau=\Z+f\cO_{\Q(\tau)}$.

Fix an imaginary quadratic field $K$. For each $f\in\N$ denote by
$\cO_{K,f}:=\Z+f\cO_K$ the order of conductor $f$ and by $K[f]/K$ the
\emph{ring class field} associated to $\cO_{K,f}$. Then $K[f]/\Q$ is a
Galois extension, and $\Gal(K[f]/K)\simeq\Pic(\cO_{K,f})$.

Denote $d:=\disc\cO_{K,f}=f^2\disc\cO_K$ and write $h(d)$ for the
\emph{class number} $h(d):=\#\Pic(\cO_{K,f})$. If $\tau\in\H$ with
$\Delta(\tau)=d$ then $K[f]=K(j(\tau))$. In particular this is
independent of the choice of $\tau$. The number of different $\tau$
satisfying $\Delta(\tau)=d$, up to $\SL_2(\Z)$-equivalence, is
$h(d)$. Moreover $\{j(\tau):\Delta(\tau)=d\}$ forms a complete set of
Galois conjugates in $K[f]/K$.

\subsection{The Siegel-Tatuzawa theorem}
\label{sec:siegel-tatuzawa}

Let $\chi_K$ denote the Dirichlet character associated to an imaginary
quadratic field $K$ and $L(s,\chi_K)$ the associated Dirichlet
L-function. The Siegel-Tatuzawa \cite{tatuzawa} theorem implies that
for any $\e>0$ we have
\begin{equation}\label{eq:siegel-tatuzawa}
  L(1,\chi_K) \ge c(\e) \disc(\cO_K)^{-\e}
\end{equation}
with the possible exception of a single imaginary quadratic field
$K_*(\e)$.

Note that, consistent with our general convention, the constant
$c(\e)$ in~\eqref{eq:siegel-tatuzawa} is effective. Without this extra
stipulation of effectivity the same statement holds for every
imaginary quadratic field $K$ by a classical result of Siegel.

Using the Dirichlet's class number formula,~\eqref{eq:siegel-tatuzawa}
implies that for any $K\neq K_*(\e)$ we have
\begin{equation}
  \#\Pic(\cO_K) \ge c(\e) \disc(\cO_K)^{1/2-\e} \qquad\text{for any }\e>0.
\end{equation}
This can be extended to an arbitrary imaginary quadratic order $\cO$
not contained in $K_*(\e)$,
\begin{equation}\label{eq:pic-vs-disc}
  \#\Pic(\cO) \ge c(\e) \disc(\cO)^{1/2-\e} \qquad\text{for any }\e>0.
\end{equation}
For a proof see \cite[Equation~(17)]{bk:linear}. We note that in this
reference the authors give an explicit constant with $\e=1/12$, but it
is clear that the proof extends for any $\e>0$.

Finally, for definiteness of our notation we set $\e_*=0.01$ and
$K_*=K(\e_*)$ .

\subsection{Galois group action for the intersection of two ring class
  fields}

Following \cite{kuhne:intersection}, for an imaginary quadratic field
$K$ we denote by $K^\tf:=\cup_{f\in\N}K[f]$ the union of the ring
class fields associated to all orders in $K$ (the notation signifies
the notion of a \emph{transfer field}). Write $L=K\cdot K_*^\tf$. Then
Cohn \cite[Proposition~8.3.12]{cohn:class-fields} proves that the
Galois group $\Gal(K^\tf\cap L/K)$ is annihilated by $2$ (i.e. has
exponent at most $2$). The same result holds for any two quadratic
fields, but we require it only with $K_*$. We remark that in
\cite{kuhne:intersection} K\"uhne proves a similar result for the
intersection of $r$ different ring class fields, with the exponent $2$
replaced by $2^{r+1}$.

The following lemma will play a key role in our argument.

\begin{Lem}\label{lem:galois-size}
  Let $\cO$ be an order of $K\neq K_*$ and set
  $L=K\cdot K_*^\tf$. Then
  \begin{equation}
    [K[\cO]\cdot L:L] \ge c(\e_*)\disc(\cO)^{1/2-\e_*}.
  \end{equation}
\end{Lem}
\begin{proof}
  We follow some arguments of \cite{kuhne:intersection}. Consider the
  following diagram of abelian field extensions.
  \begin{figure}[h]\label{fig:fields}
    \centering
    \begin{equation*}
      \xymatrix
      {
        & {K[\cO]\cdot L} & \\
        {K[\cO]} \ar@{-}[ur] & & \ar@{-}[ul] {L}\\
        & \ar@{-}[ul] {K[\cO]\cap L} \ar@{-}[ur] & \\
        & {K} \ar@{-}[u] &
      }
    \end{equation*}
  \end{figure}
  By \cite[Theorem~VI.1.12]{lang:algebra} we have
  $[K[\cO]\cdot L:L]=[K[\cO]:K[\cO]\cap L]$. Since
  $\Gal(K[\cO]/K)\simeq\Pic(\cO)$ we have by~\eqref{eq:pic-vs-disc}
  the estimate
  \begin{equation*}
    \#\Gal(K[\cO]/K) \ge c(\e_*)d^{1/2-\e_*}, \qquad d:=\disc(\cO).
  \end{equation*}
  By the result of Cohn mentioned above, $g\to2g$ induces a group
  homomorphism $\Gal(K[\cO]/K)\to\Gal(K[\cO]/K[\cO]\cap L)$. Our claim
  will thus follow once we show that the kernel, i.e. the 2-torsion
  subgroup $\Pic(\cO)[2]$ of $\Gal(K[\cO]/K)\simeq\Pic(\cO)$, has size
  at most $c(\e_*)\disc(\cO)^{\e_*}$.

  In \cite[Proposition~6.3]{zhang:equidist} it is proved that
  $\dim_{\F_2} \Pic(\cO)[2]\le 1+2\omega(d)$ where $\omega(d)$ denotes
  the number of prime divisors of $d$. Therefore we indeed have
  \begin{equation*}
    \#\Pic(\cO)[2] = 2^{\dim_{\F_2} \Pic(\cO)[2]} \le 2^{1+2\omega(d)} \le c(\e_*)d^{\e_*}
  \end{equation*}
  where we used the elementary estimate $\omega(d)\le c(\e_*)+\e_*\log_2d$.
\end{proof}

\subsection{Equidistribution of CM-points}

Let
\begin{equation}
  \cF := \{ \tau\in\H : -1/2\le\Re\tau<1/2 \text{ and
  } |\tau|>1 \} \cup \{ \Re\tau\le0 \text{ and } |\tau|=1\}
\end{equation}
denote the standard fundamental domain for the $\SL_2(\Z)$-action on
$\H$. We equip $\H$ with the invariant measure
$\d\mu(x+iy)=\frac3\pi\d x\d y/y^2$ so $\mu(\cF)=1$. For each negative
discriminant $d$ denote by $\Lambda_d$ the set of all $\tau\in\cF$
with $\Delta(\tau)=d$, so that $\#\Lambda_d=h(d)$. In
\cite{duke:equidist} Duke proves the following equidistribution
theorem for $d$ a fundamental discriminant.

\begin{Thm*}[\protect{\cite[Theorem~1.i]{duke:equidist}}]
  Suppose $\Omega\subset\cF$ is convex with a piecewise smooth
  boundary. Then for some $\delta>0$ depending only on $\Omega$,
  \begin{equation}
    \frac{\#(\Lambda_d\cap\Omega)}{\#\Lambda_d} = \mu(\Omega)+O(|d|^{-\delta})
  \end{equation}
  where the asymptotic constant depends only $\Omega$, though
  ineffectively.
\end{Thm*}

Duke's result was extended in \cite{cu:equidist} to all discriminants.
The principal source of ineffectivity is the use of Siegel's
ineffective estimate (see~\secref{sec:siegel-tatuzawa}). However, if
one restricts to discriminants associated to orders not contained in
$K_*$ one can replace this by the effective result of Siegel-Tatuzawa,
leading to an effective equidistribution result. Explicitly we will
use this result in the following form. Here and below we set
$\Omega_R:=\cF\cap\{\sqrt3/2<\Im\tau<R\}$ for any $1<R<\infty$.

\begin{Thm}\label{thm:duke-effective}
  Let $1<R<\infty$. There is an effective constant $c(R,\e_*)$ such
  that for any discriminant $d>c(R,\e_*)$ not associated to an order
  contained in $K_*$,
  \begin{equation}
    \frac{\#(\Lambda_d\cap\Omega_R)}{\#\Lambda_d} \ge 1-2A, \qquad A:=\mu(\cF\setminus\Omega_R).
  \end{equation}  
\end{Thm}

In particular Theorem~\ref{thm:duke-effective} implies that as
$R\to\infty$, the proportion of points of $\Lambda_d$ belonging to
$\Omega_R$ tends to $1$ (for discriminants $d$ satisfying the
hypoetheses). A sketch of the proof of
Theorem~\ref{thm:duke-effective} is provided in
Appendix~\ref{appendix}.

\subsection{Effective Pila-Wilkie for $Y(1)^n$}

Let $H(\alpha)$ denote the absolute multiplicative height of the
algebraic number $\alpha$. For $Z\subset\H^n$ we denote
\begin{equation}
  Z(k,H) := \{(\tau_1,\ldots,\tau_n)\in\H^n : [\Q(\tau_i):\Q]\le k, H(\tau_i)\le H \text { for } i=1,\ldots,n\}.
\end{equation}
We denote by $Z^\alg$ the union of all connected positive-dimensional
semialgebraic sets contained in $Z$ and set
$Z^\trans:=Z\setminus Z^\alg$. The Pila-Wilkie theorem
\cite{pila-wilkie}, in the variant established in
\cite{pila:algebraic-points}, implies that for any set $Z$ which is
definable in an o-minimal structure and any $\e>0$ the estimate
$\#Z^\trans(k,H)\le C(\e,k,Z)\cdot H^\e$ holds. Note however that the
constant $C(\e,k,Z)$ is ineffective, and in the vast generality of the
Pila-Wilkie theorem it is not clear in what terms one could hope to
effectively express this constant.

In \cite{me:noetherian-pw} we establish an effective version of the
Pila-Wilkie theorem for sets defined using \emph{Noetherian functions}
restricted to compact domains. We also show that $j:\H\to\C$ is
Noetherian (with effective parameters) when restricted to any compact
subset. Let $\pi:\cF^n\to Y(1)^n$ be given coordinatewise by the
$j$-function. As a consequence of \cite{me:noetherian-pw} we have the
following.

\begin{Thm}\label{thm:effective-pw}
  Let $X\subset Y(1)^n$ be an algebraic variety and $R<\infty$. Set
  $Z:=\Omega_R^n\cap\pi^{-1}(X)$. Then
  \begin{equation}
    \#Z^\trans(2,H)\le c(R,\e,\deg X)\cdot H^\e
  \end{equation}
  with an effective constant $c(R,\e,\deg X)$.
\end{Thm}

It will be important later that the constant above does not depend on
$\F_X$, nor on the heights of the coefficients of the equations
defining $X$.

\section{Estimates for positive-dimensional special subvarieties}

Let $X\subset Y(1)^n$ be an algebraic variety. Our goal in this
section is to construct a collection subvarieties
$\{X_\alpha\subset X\}$ which control $X^\sp$ in the following sense:
\begin{enumerate}
\item Each $X_\alpha$ is given up to permutation of coordinates by the
  form $\tilde X_\alpha\times V_\alpha$, where
  $\tilde X_\alpha\subset Y(1)^{k_\alpha}$ is an algebraic subvariety
  and $V_\alpha\subset Y(1)^{n-k_\alpha}$ is a strongly special
  variety.
\item Each maximal special subvariety $V\in X^\sp$ is of the form
  $\{P\}\times V_\alpha$ with $P\in\tilde X_\alpha^\spp$, for some $\alpha$.
\end{enumerate}
It is clear that constructing such a collection reduces the
computation of $X^\sp$ to the computation of
$\tilde X_\alpha^\spp$. The existence of finite collections of this
type follows easily from the results of \cite{pila:andre-oort}. Our
goal, more explicitly, is to obtain such collections
effectively. Specifically we have the following result.
  
\begin{Thm}\label{thm:Xsp-dim+}
  Let $X\subset Y(1)^n$ be an algebraic variety. There exists a
  collection $\{X_\alpha\subset X\}$ as above, with the number of
  subvarieties $X_\alpha$ and their degrees bounded by
  $\poly_n(\deg X)$.
\end{Thm}

We denote by $X^\wsp$ the union of all positive-dimensional
weakly-special subvarieties of $X$. The results of
\cite{pila:andre-oort} imply that $X^\wsp$ is Zariski closed, but do
not yield an effective estimate on its degree. We will reduce the
proof of Theorem~\ref{thm:Xsp-dim+} to the following lemma, whose
proof occupies the remainder of this section.

\begin{Lem}\label{lem:qsp-bound}
  Let $X\subset Y(1)^n$ be an algebraic variety. Then
  \begin{equation}
    \deg X^\wsp \le c(n)(\deg X)^{16n^2}.
  \end{equation}
\end{Lem}

Our proof of Lemma~\ref{lem:qsp-bound} is based on methods of
differential equations. We begin by demonstrating a general result
providing effective bounds for the degree of the collection of
trajectories of a rational vector field that belong to a given
algebraic variety. We then describe the graph of $j(\tau)$ (and its
first two derivatives) as trajectories of a rational vector
field. Finally, we apply this result in combination with the
functional independence results of Pila \cite{pila:andre-oort} to
obtain an effective description of $X^\wsp$.

We now show how Lemma~\ref{lem:qsp-bound} implies
Theorem~\ref{thm:Xsp-dim+}.

\begin{proof}[Proof of Theorem~\ref{thm:Xsp-dim+}]
  We proceed by induction on $\dim X$.  Let $V\in X^\sp$. Then
  certainly $V\subset X^\wsp$, and replacing $X$ by the irreducible
  component of $X^\wsp$ that contains $V$ we may assume without loss
  of generality that $X$ is irreducible and $X=X^\wsp$ (the crucial
  bound on the degree in this reduction follows from
  Lemma~\ref{lem:qsp-bound}).

  Every point of $X$ is contained in a weakly-special subvariety. Every
  weakly-special subvariety $W\subset X$ satisfies one of the following
  conditions:
  \begin{enumerate}
  \item Up to reordering the coordinates, $W$ is of the form
    $W'\times Y(1)$.
  \item The equation $\Phi_N(x_i,x_j)$ vanishes identically on $W$,
    for some $N\in\N$ and $i\neq j$.
  \end{enumerate}
  If $W\subset X$ satisfies condition 1 then $X$ contains the
  trajectory of $\pd{}{x_i}$ (for some $i$) through every point of
  $W$. The set of points where this happens is easily seen to be
  Zariski closed (see e.g. Lemma~\ref{lem:red-degree}, although this
  case is completely elementary). The zeros of $\Phi_N(x_i,x_j)$ are
  certainly also Zariski closed. We thus see that $X$ is the union of
  a countable collection of Zariski closed subsets, and since $X$ is
  irrediclbe we conclude that in fact one of the conditions 1--2 holds
  identically with $W$ replaced by $X$.

  Suppose that the second conditions is satisfied identically on
  $X$. Up to reordering the coordinates we may suppose that
  $(i,j)=(n-1,n)$. Then the projection to first $n-1$ coordinates
  gives a finite map $\pi:X\to X'$ where $X'=\pi(X)$, and the degree
  of this map is bounded above by $\deg X$. We apply the inductive
  hypothesis to $X'$ to obtain a collection $\{X'_\beta\}$. We define
  the collection $\{X_\alpha\}$ to be the collection of irreducible
  components of $\pi^{-1}(X'_\beta)$ for all $\beta$. Since
  $\Phi_N(x_n,x_{n-1})$ vanishes identically on $X$ it follows that
  indeed each such component is of the required special form (where
  $x_n$ belongs to the same ``block'' as $x_{n-1}$). Since
  $\pi(V)\subset X'$ is special it follows by induction that it
  belongs to some $X'_\beta$, and consequently $V$ belongs to some
  $X_\alpha$ as required.

  If $X$ is of the form $X=X'\times Y(1)$ up to reordering the
  coordinates then the claim follows by a similar (but simpler)
  induction over dimension for $X'$.
\end{proof}

\subsection{The collection of trajectories contained in an algebraic
  variety}

Let $M=\C^m\setminus\Sigma$ where $\Sigma$ is an algebraic
hypersurface, and let $\xi$ be a vector field whose coefficients are
regular functions on $M$. Let $W$ be a subvariety of $M$. Denote by
$\cR_\xi W$ the union of all trajectories of $\xi$ that are contained
in $W$. We have the following.

\begin{Lem}\label{lem:red-degree}
  In the notation above, $\cR_\xi W\subset M$ is an algebraic
  subvariety. Its Zariski closure in $\C^m$ has degree at most
  $c(\xi)\deg(W)^{m^2}$.
\end{Lem}
\begin{proof}
  Replacing $\xi$ by $f^N\xi$ where $\Sigma=(f)$ we may suppose that
  $\xi$ is a polynomial vector field (note this scalar multiplication
  does not affect the trajectory structure of $\xi$, only the time
  parametrization). Let $\{F_\alpha\}$ be a collection of polynomial
  equations of degrees at most $\deg W$ which define $W$
  set-theoretically. The condition $p\in\cR_\xi W$ is equivalent to
  the condition the derivatives $\xi^kF_\alpha$ vanish for every
  $k\in\N$ and every $\alpha$.

  It is possible to give an effective upper bound for the maximal
  order of vanishing of a polynomial along the trajectory of a
  polynomial vector field (assuming that the polynomial does not
  vanish identically on the trajectory). Such results are known as
  \emph{multiplicity estimates}, see
  e.g. \cite{nesterenko:mult-nonlinear,gabrielov:mult,me:mult-morse}. The
  sharpest such estimate presently known is as follows.
  \begin{Thm*}[\protect{\cite[Corolary~1]{me:mult-morse}}]
    Let $p$ be a non-singular point of $\xi$ and $F$ be a polynomial
    of degree $d$ in $m$ variables. Then the multiplicity of the zero
    of $F$ when restricted to the trajectory of $\xi$ through $p$,
    assuming it is finite, does not exceed
    \begin{equation}
      \mu = 2^{m+1}(d+(m-1)(\delta-1))^m
    \end{equation}
    where $\delta$ is the degree of the vector field $\xi$.
  \end{Thm*}
  Applying this to $F_\alpha$ we see that the vanishing of
  $\xi^kF_\alpha$ for every $k$ is equivalent to the vanishing of the
  $k=1,\ldots,\mu$ derivatives where $\mu=c(\xi)d^m$ for
  $d=\deg W$. We thus have a system of polynomials of degrees at most
  $c(\xi)d^m$ which define $\cR_\xi W$ set-theoretically. From this it
  is straightforward to deduce using the Bezout theorem that the
  degree of the Zariski closure of $\cR_\xi W$ in $\C^m$ is bounded as
  claimed.
\end{proof}

\subsection{The $j$-function as a trajectory of a rational vector
  field}
\label{sec:j-diff-eq}

Recall that the Schwartzian operator is defined by
\begin{equation}
  S(f) = \left(\frac{f''}{f'}\right)' - \frac12 \left(\frac{f''}{f'}\right)^2
\end{equation}
We introduce the differential operator
\begin{equation}
  \chi(f) = S(f) + R(f) (f')^2, \qquad R(f) = \frac{f^2-1968f+2654208}{2f^2(f-1728)^2}
\end{equation}
which is a third order algebraic differential operator vanishing on
Klein's j-invariant $j$ \cite[Page~20]{masser:heights}. As observed in
\cite{fs:minimality} it easy to check that the solutions of
$\chi(f)=0$ are exactly the functions of the form
$j_g(\tau):=j(g\cdot\tau)$ where $g\in\PGL_2(\C)$ acts on $\C$ in the
standard manner.

The differential equation above may be written in the form
$f'''=A(f,f',f'')$ where $A$ is a rational function. More explicitly, 
consider the ambient space $M:=\C\times\C^3\setminus\Sigma$ with
coordinates $(\tau,y,\dot y,\ddot y)$ where $\Sigma$ consists of the
zero loci of $y,y-1728$ and $\dot y$. On this space the vector field
\begin{equation}
  \xi := \pd{}\tau + \dot y \pd{}y + \ddot y\pd{}{\dot y}+A(y,\dot y,\ddot y)\pd{}{\ddot y}
\end{equation}
encodes the differential equation above, in the sense that any
trajectory is given by the graph of a function $j_g(\tau)$ and its
first two derivatives. We remark that $0,1728$ play a special role as
the critical values of the $j$-function.

\subsection{Description of $X^\wsp$ using vector fields}

We will use the following characterization of the weakly-special
varieties in terms of the $j$-function
\cite[Definition~6.7]{pila:andre-oort}
(cf.~\eqref{eq:special-variety}). Let $S_0\cup\cdots\cup S_w$ be a
partition of $\{1,\ldots,n\}$ with $S_0$ only permitted to be
empty. Let $t_i\in\H$ for $i\in S_0$. Let $s_i\in S_i$ be minimal for
$i>0$ and for each $s_i\neq j\in S_i$ choose an element
$g_{ij}\in\GL_2^+(\Q)$. Then $V$ is a $w$-dimensional weakly-special
variety if an only if it is the image under $j$ of the set
\begin{multline}\label{eq:special-variety-H}
  \{ (\tau_1,\ldots,\tau_n)\in\H^n : \tau_i=t_i, i\in S_0,\\
  y_j=g_{i,j}\cdot y_{s_i}, s_i\neq j\in S_i, i=1,\ldots, w\}.
\end{multline}
A special variety is obtained exactly when each $t_i$ for $i\in S_0$
is quadratic.

Let $S\subset\{1,\ldots,n\}$ of size $q$. We consider the ambient space
\begin{equation}
  \Omega = \C_\tau\times\Omega_1\cdots\times\Omega_n
\end{equation}
where for $i\in S$ we set $\Omega_i=\C^3\setminus\Sigma$ with $\Sigma$
as in~\secref{sec:j-diff-eq}; and for $i\not\in S$ we set
$\Omega_i=\C$. We use the coordinates $\tau$, $(y_i)_{1\le i\le n}$
and $(\dot y,\ddot y)_{i\in S}$. On this space we consider the vector
field $\xi$ given on the $\Omega_i$ with $i\in S$ factors as
in~\secref{sec:j-diff-eq}, and on remaining $\Omega_i$ factors as
zero.

The trajectories of $\xi$ are given by function of the following
form. For each $i\in S$ fix some $g_i\in\PGL_2(\C)$ and for each
$i\not\in S$ fix some $p_i\in\C$. Then the corresponding solution is
the function $\tau\to(\tau,\gamma(\tau))$ where $\gamma$ is given in
the $i\in S$ coordinates by $j(g_i\cdot\tau)$ and in the $i\not\in S$
coordinates by the constant $p_i$.

Let $X\subset Y(1)^n$ which we identify with $\C^n$ with coordinates
$y_1,\ldots,y_n$ using the $j$-function. Let $W=\pi^{-1}(X)$ where
$\pi:\Omega\to\C^n$ is the projection map. We claim that the Zariski
closure of $\pi(\cR_\xi W)$ is a subset of $X^\wsp$. Since $X^\wsp$ is
Zariski closed it is enough to prove that 
$\pi(\cR_\xi W)\subset X^\wsp$.

Let $P\in\pi(\cR_\xi W)$ and let $\gamma$ be the trajectory contained
in $W$ with $P\in\pi(\gamma)$. This means that $V$ contains a curve
given by $y_i=j(g_i\cdot\tau)$ in the $i\in S$ coordinate and by a
constant in the remaining coordinates (for $\tau$ in some open
neighborhood where the expressions above are all defined). Therefore
the preimage of $V$ in $\H^n$ (by coordinate-wise application of $j$)
contains an algebraic curve given by $\tau_i=g_i^{-1}g_j\tau_j$ for
$i,j\in S$ and by a constant on the remaining coordinates. By a result
of Pila \cite[Theorem~6.8]{pila:andre-oort} any such algebraic curve
belongs to the pre-image of a weakly-special variety contained in
$V$. Thus $P\in X^\wsp$ as claimed.

We now claim that the union of the above constructed $\pi(\cR_\xi W)$
over every choice of $S$ contains $X^\wsp$. Indeed, let
$P\in X^\wsp$. Then $X$ contains some positive-dimensional
weakly-special variety $V$ containing $p$. Let $S$ be given by the
block $S_1$ in the definition of $V$ as
in~\eqref{eq:special-variety-H}. Choose $g_{s_1}=\id$ and
$g_j=g_{1,j}$ for $s_1\neq j\in S_1$. For $i\not\in S$ choose
$p_i=P_i$. Then the trajectory $\gamma$ corresponding to this data is
contained in $W$ by definition and its projection $\pi(\gamma)$ passes
through $P$. Thus $P\in\pi(\cR_\xi W)$. The proof is now concluded by
application of Lemma~\ref{lem:red-degree} for each choice of $S$.

\section{Proofs of the main results}

In this section we give the proofs of Theorem~\ref{thm:main} and
consequently of Corollaries~\ref{cor:no-modular},~\ref{cor:dnd}
and~\ref{cor:hdnd}.

\subsection{Proof of Theorem~\ref{thm:main}}

Using Theorem~\ref{thm:Xsp-dim+} we already have the required estimate
for $\deg V$, and the general problem is reduced to producing the
requisite upper bounds for $D_*(P)$ and $D(P)$ where $P\in X^\spp$. We
fix such $P$ and denote $K_i=K_i(P),\Delta_i=\Delta_i(P)$ and
$\cO_i=\cO_i(P_i)$.

We argue first assuming $\F_X=\Q$, indicating the very minor changes
needed for the general case at the end. Denote by
$O_P:=\Gal(\bar\Q/\Q)\cdot P$ the set of all Galois conjugates of
$P$. By our assumption $O_P\subset X$. To simplify our notations we
reorder the coordinates so that $K_1,\ldots,K_m\neq K_*$ and
$K_{m+1},\ldots,K_n=K_*$. Choose $R$ so that the
$\mu(\Omega_R)\ge 1-1/(4m)$. We note that the Galois action on $O_P$
factors as a transitive action on $\Lambda_{\Delta_i}$ when restricted
to the $i$-th coordinate. Theorem~\ref{thm:duke-effective} for each
$i=1,\ldots,m$ thus implies that if $\Delta_i(P)>c(R,\e_*)$ then
\begin{equation}\label{eq:Op-equi-basic}
  \#(\O_P\cap\{P_i\in j(\Omega_R)\})\ge (1-1/2m)\cdot\#O_P.
\end{equation}
Increasing $R$ if necessary so that $\Omega_R$ also contains
\emph{every} point $\tau\in\H$ of discriminant smaller than the
(effective) constant $c(R,\e_*)$ above, we can in fact assume
that~\eqref{eq:Op-equi-basic} holds without the restriction on
$\Delta_i(P)$. We set $\Omega=\Omega_R^m$ for this $R$ and let
$\pi:\cF^m\to Y(1)^m$ be given coordinatewise by the $j$-function.

Taking intersection of~\eqref{eq:Op-equi-basic} for $j=1,\ldots,m$ we
see that
\begin{equation}\label{eq:Op-equi}
  \#(\O_P\cap[\pi(\Omega)\times Y(1)^{n-m}])\ge (1/2) \cdot\#O_P.
\end{equation}
By an averaging argument over the fibers of the projection to
$Y(1)^{n-m}$ we see that there exists a point
$Q=(Q_1,\ldots,Q_{n-m})\in K_*^\tf$ such that, with
$Y_Q:=Y(1)^m\times\{Q\}$, we have $O_P\cap Y_Q\neq\emptyset$ and
\begin{equation}
  \#(O_P\cap[\pi(\Omega)\times\{Q\}])\ge (1/2) \cdot\#(O_P\cap Y_Q).
\end{equation}
For each $i=1,\ldots,m$ let $L_i:=K_i\cdot K_*^\tf$. The Galois group
$\Gal(K_i[\cO_i]\cdot L_i/L_i]$ fixes $Q$ and hence acts on the set
$O_P\cap Y_Q$. Moreover its action is faithful on the
$i$-th coordinate since $K_i[\cO_i]$ is generated over $K_i$ by
$P_i$. It follows that
\begin{multline}\label{Op-Q-estimate}
  \#(O_P\cap[\pi(\Omega)\times\{Q\}])\ge (1/2) \cdot\#(O_P\cap Y_Q) \\
  \ge (1/2)\cdot\#\Gal(K_i[\cO_i]\cdot L_i/L_i) \ge c(\e_*)|\Delta_i|^{1/2-\e_*}
\end{multline}
where the final inequality follows from Lemma~\ref{lem:galois-size}.

Let $\Delta$ denote the maximal discriminant among
$\Delta_1,\ldots,\Delta_m$. Set $X_Q:=X\cap Y_Q$, which we view as a
subvariety of $Y(1)^m$ in the obvious way, and let
$Z:=\pi^{-1}(X_Q)$. Each point of $O_P\cap Y_Q$ corresponds to a point
of $X_Q$, and moreover since these points are Galois conjugate to $P$
none of them are contained in a special subvariety of positive
dimension in $X_Q$ (since this would also be a special subvariety of
$X$). According to \cite[Theorem~6.8]{pila:andre-oort}, the
$\pi$-preimage of each such point belongs to $Z^\trans$. According to
the proof of \cite[Proposition~5.7]{pila:andre-oort} they each have
height at most $\const\cdot\Delta$. Combined
with~\eqref{Op-Q-estimate} we conclude that
\begin{equation}
  \#(Z\cap\Omega)^\trans(2,\const\cdot|\Delta|) \ge c(\e_*)|\Delta|^{1/2-\e_*}.
\end{equation}
On the other hand, Theorem~\ref{thm:effective-pw} implies that
\begin{equation}
  \#(Z\cap\Omega)^\trans(2,\const\cdot|\Delta|) \le c(\e,\deg X)|\Delta|^\e.
\end{equation}
For say $\e=1/3$ these two competing estimates imply that
$|\Delta|<c(\deg X)$ as claimed.

If we allow our constants to depend on $\Delta_*$ then both
Lemma~\ref{lem:galois-size} and Theorem~\ref{thm:duke-effective}
become effective without any restriction on the discriminant, and the
proof above yields an estimate on $\Delta_i$ with no restriction on
$K_i$.

If $\F_X$ is an arbitrary (say normal) number field then we replace
$X$ by the union of its $\Gal(\bar\Q/\Q)$-conjugates, which is a
variety defined over $\Q$ and of degree $[\F_X:\Q]\cdot\deg X$. The
claim now follows from what was already proved (since if $P$ is not
contained in a special variety of positive dimension in $X$ then the
same is true for each of its conjugates).

\subsection{Proof of Corollary~\ref{cor:no-modular}}

The proof is similar to that of \cite[Corollary~4]{kuhne:ao2}. We
reduce to the case $\F_X=\Q$ as above. Further, using
Theorem~\ref{thm:Xsp-dim+} we immediately reduce from the case of
general $V$ to the case $V=\{P\}$ for some $P\in X^\spp$.

We first prove that the number of points $P$ satisfying the conditions
of the corollary is bounded by $c(\deg X,[\F_X:\Q])$. To see this, we
reduce the computation of $X^\spp$ to the computation of $*$-points in
$X^\spp_\alpha$ for the collection of special sections $\{X_\alpha\}$
described following the proof of Theorem~\ref{thm:main}. It is enough
to show that the number of points satisfying the conditions of the
corollary in each of these special sections is bounded by such a
constant.

If the ambient dimension of $X_\alpha$ is greater than one then none
of the $*$-points in $X_\alpha$ correspond to the points considered in
the corollary. If the ambient dimension of $X_\alpha$ is zero then it
correponds to a single point. Finally, suppose the ambient dimension
is one. If $X_\alpha$ is zero-dimensional then the total number of
points in $X_\alpha$ is bounded by $\deg X$. If $X_\alpha=Y(1)$ then
$X_\alpha^\spp=\emptyset$. 

Now let $P$ be as in the conditions of the corollary. From the above
we conclude that for $i=1,\ldots,n$, the number of Galois conjugates
of $P_i$ should be bounded by $c(\deg X,[\F_X:\Q])$. In
other words, $h(\Delta_i)<c(\deg X,[\F_X:\Q])$. The effective bound on
$|\Delta_i|$ now follows from the deep effective estimate
\begin{equation}
  h(d) \ge c(\e) (\log |d|)^{1-\e}, \qquad \forall \e>0
\end{equation}
due to Goldfeld-Gross-Zagier \cite{goldfeld,gz:heegner}.

\subsection{Proof of Corollary~\ref{cor:dnd}}

Note that we may ssume that the polynomial defining $X$ depends on all
variables, since otherwise $X^\spp$ is empty. Suppose that we are in
the case that the $P_1,\ldots,P_n$ are pairwise distinct. Suppose
without loss of generality that $\Delta_1$ is maximal among
$\Delta_i$. If $K_1\neq K_*$ then Theorem~\ref{thm:main} gives an
upper bound for $|\Delta_1|$ and we are done. Therefore assume that
$K_1=K_*$ and write $\Delta_1=f^2\Delta_*$. Following \cite{bk:linear}
we note that for every discriminant $\Delta$,
\begin{equation}
  \tau_\Delta := \frac{-b_\Delta+i\sqrt{|\Delta|}}2 \qquad \{0,1\}\ni b_\Delta\equiv\Delta\mod 2
\end{equation}
is a CM-period of discriminant $\Delta$, and moreover every other
CM-period of discriminant $\Delta$ in $\cF$ has imaginary part at most
$\Im\tau_\Delta/2$. Recall the following estimate from
\cite[Lemma~1]{bmz:ao}:
\begin{equation}
  | |j(\tau)|-e^{2\pi\Im\tau} | < 2079 \qquad \text{for every }\tau\in\bar\cF.
\end{equation}
After applying a Galois conjugation we may assume that
$P_1=j(\tau_{\Delta_1})$. From the above we deduce that
\begin{align*}
  |P_1| &\ge \const\cdot e^{\pi f |\Delta_*|^{1/2}}, \\
  |P_i| &\le c(\deg X,[\cF_X:\Q])\cdot e^{\pi(f-1)|\Delta_*|^{1/2}} \text{ for }i=2,\ldots,n.
\end{align*}
Indeed, for $K_i\neq K_*$ the estimate follows from
Theorem~\ref{thm:main}. For $K_i=K_*$, if $\Delta_i=\Delta_1$ then
since these points are distinct we have
$\Im\tau_i\le(1/2)\Im\tau_\Delta$; and otherwise
$\Delta_i\le(f-1)^2\Delta_*$. In particular we deduce that
\begin{equation}
  \frac{|P_1|}{|P_i|} \ge c(\deg X,[\cF_X:\Q]) \cdot e^{\pi |\Delta_*|^{1/2}} 
\end{equation}
for $i=2,\ldots,n$.

Let $c_0$ denote the (non-zero) absolute value of the coefficient of
$x_1^d$ in the polynomial $F$ defining $X$, and let $c_1$ denote the
maximum of the absolute values of all remaining coefficients.  It is
clear that for $|\Delta_*|$ larger than some effective function of
$c_0,c_1$ and $d$, the term $P_1^d$ becomes dominant and the equation
$F(P)=0$ cannot be satisfied. On the other hand if
$\Delta_*<c(\deg X,H(X),[\F_X:\Q])$ for such an effective function
then Theorem~\ref{thm:main} is already effective for $K_*$ as well,
thus concluding the proof.

\subsection{Proof of Corollary~\ref{cor:hdnd}}

We first show that any special subvariety $V\subset X$ is
linear. Suppose $V$ is positive-dimensional. If $V$ is not linear then
it contains a special curve which is not linear, so we may assume that
$V$ is a curve. If a coordinate function $x_i$ is contant on $V$ then
we may reduce to corresponding section of $X$ (note that $X$ remains
hdnd upon taking such a section). Similarly if $x_i-x_j\equiv0$ on $V$
then we may pass to the diagonal $x_i=x_j$. Afrer making these
reductions we may assume that each $x_i$ is non-constant on $V$ and
each $x_i-x_j$ vanishes at a finite number of points on $V$.

Let $K$ be a quadratic field. Since $V$ is a special curve with no
constant coordinates it consists of a single block $S_1$, and it
follows that $V$ contains a point $P$ where every coordinate $x_j$ is
a special point associated to $K$. Moreover by the finiteness of the
zeros of $x_i-x_j$ on $V$ we see that excluding a finite number of
fields, the coordinates of $P$ are distinct. The proof of
Corollary~\ref{cor:dnd} shows that a dnd hypersurface cannot contain
points of this type for a sufficiently large fundamental discriminant,
and we thus obtain a contradiction for an appropriate choice of $K$.

We now prove the estimate on $D(V)$. Since $V$ is linear it is given
up to reordering the coordinates by $\{P\}\times V'$ where
$P\in Y(1)^m$ and $V'\subset Y(1)^{n-m}$ is a strongly special linear
variety. In particular $V'$ contains some universally fixed special
point $Q$, say $Q:=(j(i),\ldots,j(i))$. Thus $P\in\tilde X^\spp$ where
$\tilde X$ is the section obtained by setting
$(x_{n-m+1},\ldots,x_n)=Q$. Furthermore $\tilde X\subset Y(1)^m$ is
again given by an hdnd polynomial.

The above reduces the problem of estimating $D(V)$ to the problem of
estimating $D(P)$ for $P\in X^\spp$. If the coordinates of $P$ are
pairwise distinct this follows from Corollary~\ref{cor:dnd}. If
$P_i=P_j$ then it follows by induction over dimension reducing to the
diagonal $x_i=x_j$ (noting again that $X$ remains the hdnd on the
diagonal by definition).

\appendix

\section{Duke's Theorem avoiding Tatuzawa fields (by~E.~Kowalski)}
\label{appendix}

\def\eps{\e}
\def\Hh{\H}
\def\Zz{\Z}
\def\Imag{\Im}
\def\Cc{\C}
\def\Rr{\R}

Duke's Theorem~\cite[Th. 1]{duke:equidist} for CM points states that
as $d\to+\infty$, the CM points $\Lambda_d$ of discriminant $-d$
become equidistributed in the modular curve $Y(1)$, in the
quantitative form
$$
\frac{|\Lambda_d\cap \Omega|}{|\Lambda_d|}=\mu(\Omega)+O(|d|^{-\delta})
$$
for some $\delta>0$ depending only on the domain $\Omega\subset Y(1)$,
which is assumed to be convex with piecewise smooth boundary. Here
$\mu$ is the hyperbolic area measure normalized so that $\mu(Y(1))=1$.
The constant $\delta>0$ and the implied constant depend only on
$\Omega$; the former is effective (and explicit), but the latter is
not, due to the use of Siegel's lower bound for class numbers in the
proof.

Tatuzawa~\cite[Th. 1]{tatuzawa} has given a form of Siegel's bound that,
for a given $\eps>0$, gives an effective (explicit) constant
$c(\eps)>0$ (in fact, proportional to $\eps$) such that
$$
|\Lambda_d|> c(\eps)|d|^{1/2-\eps}
$$
for all fundamental discriminants $d$ with at most one exception,
which may depend on $\eps$. We write $d_{\eps}^*$ for this exception.

Let $F$ be the standard fundamental domain of
$\SL_2(\Zz)\backslash \Hh$. We view $\Lambda_d$ as a subset of $F$.
We consider the regions
$$
\Omega_R=\{z\in F\,\mid\, \sqrt{3}/2<\Imag(z)<R\}.
$$
Note a minor issue even in a direct application of Duke's Theorem,
without consideration of effectivity: this region is \emph{not} convex
in the hyperbolic sense.

The following is however a simple deduction from the principles of the
proof, combined with Tatuzawa's Theorem.

\begin{Prop}
  Let $0<\eps<1/16$ be fixed. Let $d_{\eps}^*$ be the corresponding
  exceptional discriminant. Let $m\geq 1$ be an integer. There exists
  an effective constant $c(m,\eps)>0$ such that for $d>c(m,\eps)$ and
  for $d\not=d_{\eps}^*$, we have
$$
\frac{1}{|\Lambda_d|}|\{z\in\Lambda_d\,\mid\, z\in \Omega_{8m}\}|\geq
1-\frac{1}{4m}.
$$
\end{Prop}

\begin{Rem}
(1)  The choice $R=8m$ is to have the hyperbolic area of $\Omega_R$ equal
  to $1-1/R$.  
\par
(2) We sketch the proof with fundamental discriminants, but proceeding
as in Clozel--Ullmo \cite{cu:equidist}, it extends to all
discriminants.
\end{Rem}

\begin{proof}[Sketch of proof]
  Recall that
  $$
  |\Lambda_d|=\frac{1}{2\pi}w_d|d|^{1/2}L(1,\chi_d)
  $$
  where $w_d$ is the number of roots of unity in the quadratic field
  with discriminant $d$, by Dirichlet's Class Number Formula (see,
  e.g.,~\cite[(22.59)]{iwaniec-kowalski}).
  
  Let $R=8m$.  Consider a smooth compactly supported function
  $\psi\colon Y(1)\to \Cc$ which satisfies $0\leq \psi\leq 1$, is
  equal to $1$ for $z\in F$ with imaginary part $\leq R/2$ (hence on
  $\Omega_{R/2}$) and vanishes for $z$ with imaginary part $\geq R$,
  and such that the partial derivatives of $\psi$ are bounded (by
  constants depending only on the order of the derivative).

  Note that the choice of such a function depends on $m$. Observe that
$$
\frac{1}{|\Lambda_d|}|\{z\in\Lambda_d\,\mid\, z\in \Omega_{R}\}| \geq
\frac{1}{|\Lambda_d|}\sum_{z\in\Lambda_d} \psi(z).
$$
Now by the spectral decomposition in $L^2(Y(1))$ (see,
e.g,~\cite[Th. 4.7, Th. 7.3]{iwaniec}), we have
$$
\psi(z)=\mu_{\psi}+\int_{0}^{+\infty}\langle
\psi,E(\cdot,1/2+it)\rangle E(z,1/2+it)dt+\sum_j \langle \psi,
u_j\rangle u_j(z)
$$
where
$$
\mu_{\psi}=\int_{F}\psi(z)d\mu(z),
$$
the functions $E(z,s)$ are the Eisenstein series for $\SL_2(\Zz)$ and
$(u_j)$ runs over an orthonormal basis of the cuspidal subspace of
$L^2(Y(1))$, which we may assume consists of Hecke eigenforms.
\par
We have
$$
\mu_{\psi}\geq \mu(\Omega_{R/2})=1-\frac{2}{R},
$$
hence
$$
\frac{1}{|\Lambda_d|}\sum_{z\in\Lambda_d} \psi(z)\geq
1-\frac{1}{4m}+\mathcal{R}
$$
where
$$
\mathcal{R}=\int_{0}^{+\infty}\langle \psi,E(\cdot,1/2+it)\rangle
\frac{1}{|\Lambda_d|}\sum_{z\in\Lambda_d}E(z,1/2+it)dt +\sum_j \langle
\psi,u_j\rangle \frac{1}{|\Lambda_d|}\sum_{z\in\Lambda_d}u_j(z).
$$
\par
A classical formula (see references in~\cite[p. 88]{duke:equidist}
or~\cite[(22.45)]{iwaniec-kowalski}) computes
$$
\frac{1}{|\Lambda_d|} \sum_{z\in\Lambda_d}E(z,1/2+it)=w_d
\frac{\zeta(1/2+it)L(\chi_d,1/2+it)}{|d|^{1/4-it/2}L(1,\chi_d)}
$$
where $w_d$ is the number of roots of unity in the quadratic field.
Combining an old result of Weyl for $\zeta(s)$ and a result of
Heath-Brown~\cite{HB}, whose proof is effective, yields upper bounds
$$
\zeta(1/2+it)\ll (1+|t|)^{1/6},\quad\quad L(\chi_d,1/2+it)\ll
|d|^{1/6+\eta}(1+|t|)^{1/6+\eta}
$$
for any $\eta>0$, where the implied constant is effective and depends
only on $\eta$.
\par
On the other hand, using the Waldspurger formula (see the discussion
of Michel and Venkatesh~\cite[(2.5)]{mv2}), one finds a formula of the
(similar) type
$$
\Bigl|\frac{1}{|\Lambda_d|}\sum_{z\in\Lambda_d}u_j(z)\Bigr|^2=\alpha
\frac{L(u_j,1/2)L(u_j\times \chi_d,1/2)}{|d|^{1/2}L(1,\chi_d)^2}
$$
(where $\alpha$ is a constant) in terms of central values of twisted
$L$-functions.  We use the subconvexity estimate of Blomer and
Harcos~\cite[Th. 2]{bh} (although we could use also that of Michel and
Venkatesh~\cite{mv}, or indeed any subconvex bound that has polynomial
control in terms of the eigenvalue of $u_j$ would suffice, and there
are many more versions): we have
$$
L(u_j\times \chi_d,1/2)\ll |d|^{3/8}(1+|t_j|)^3,\quad\quad
L(u_j,1/2)\ll (1+|t_j|)^3,
$$
where the implied constants are effective and $1/4+t_j^2$ is the
Laplace eigenvalue of the cusp form $u_j$ (we have $t_j\in\Rr$ since
it is known that there are no eigenvalues $<1/4$ for $Y(1)$).

Using ``integration by parts'', namely writing
$$
\langle \psi,u_j\rangle=\frac{1}{(1/4+t_j^2)^A}
\langle \psi,\Delta^A u_j\rangle=
\frac{1}{(1/4+t_j^2)^A}
\langle \Delta^A\psi,u_j\rangle,
$$
we obtain for any $A\geq 1$ the bound
$$
|\langle \psi,u_j\rangle|\leq \frac{1}{(1/4+t_j^2)^A}\|\Delta^A\psi\|
\ll_A \frac{R^{2A+1}}{(1+|t_j|)^{2A}}
$$
(since $\Delta^A \psi(z)$ vanishes unless $R/2\leq \Imag(z)\leq R$,
and the derivatives are bounded).  Similarly, one gets
$$
\langle \psi,E(\cdot,1/2+it)\rangle\ll \frac{R^{2A+1}}{(1+|t|)^{2A}}.
$$
for any $A>0$, where the implied constant depends on $A$ and is
effective.
\par
Taking $A$ fixed and large enough to make the integral and series
converge absolutely (e.g., $A=3$), we derive the lower bound
$$
\frac{1}{|\Lambda_d|}|\{z\in\Lambda_d\,\mid\, z\in \Omega_{8m}\}| \geq
1-\frac{2}{R}+O\Bigl(R^{2A+1}|d|^{1/2-1/16}|\Lambda_d|^{-1} \Bigr),
$$
where the implied constant is effective, and hence  for
$d\not=d^*_{\eps}$, we obtain
$$
\frac{1}{|\Lambda_d|}|\{z\in\Lambda_d\,\mid\, z\in \Omega_{R}\}| \geq
1-\frac{2}{R}+O\Bigl(R^{2A+1}|d|^{\eps-1/16} \Bigr),
$$
where the implied constant is effective. The result now follows.
\end{proof}

\bibliographystyle{plain} \bibliography{nrefs}

\end{document}